\documentclass[
11pt]{amsart}

\usepackage{amsthm, amsfonts, amssymb, xypic}
\usepackage[all]{xy}
\usepackage{amscd}
\usepackage[pagebackref,colorlinks]{hyperref}

\theoremstyle{definition}
\newtheorem{ntn}{Notation}[section]

\theoremstyle{plain}
\newtheorem{lem}[ntn]{Lemma}

\newtheorem{prp}[ntn]{Proposition}
\newtheorem{thm}[ntn]{Theorem}

\theoremstyle{remark}

\newtheorem{rem}[ntn]{Remark}

\def\floor[#1]{\lfloor #1 \rfloor }

\newcommand{\z}{\mathbb{Z}}

\newcommand{\lan}{\langle}
\newcommand{\ran}{\rangle}

\newcommand{\GL}{\mathit{{\rm GL}}}

\newcommand{\SL}{\mathit{{\rm SL}}}

\newcommand{\E}{\mathcal{E}}

\newcommand{\DD}{{\mathfrak{d}}}
\newcommand{\ppp}{\mathfrak{p}}

\renewcommand{\H}{\tilde{H}}
\newcommand{\ee}{\mathcal{E}}

\newcommand{\inc}{{\rm inc}}
\newcommand{\id}{{\rm id}}

\newcommand{\tors}{{{\rm Tor}_1^{\z}}}

\newcommand{\si}{\sigma}

\newcommand{\arr}{\rightarrow}
\newcommand{\larr}{\longrightarrow}

\newcommand{\se}{\subseteq}

\newcommand{\mt}{\mapsto}

\newcommand{\fff}{{F^\ast}}

\renewcommand{\char}{{\rm char}}
\newcommand{\diag}{{\rm diag}}
\renewcommand{\ker}{{\rm ker}}
\newcommand{\coker}{{\rm coker}}
\newcommand{\im}{{\rm im}}
\newcommand{\ind}{{\rm ind}}

\newcommand {\mtx}[4]
{\left(
\begin{array}{cc}
#1 & #2   \\
#3 & #4
\end{array}
\right)}

\newtheoremstyle{athm}
  {}
  {}
  {\itshape}
  {}
  {\scshape}
  {}
  {.5em}
  {\thmnote{#3}}
\theoremstyle{athm}
\newtheorem*{athm}{}

\begin{document}

\title{Third homology of $\SL_2$ and the indecomposable $K_3$}
\author{Behrooz Mirzaii}
\begin{abstract}
It is known that, for an infinite field $F$, the indecomposable part
of $K_3(F)$ and the third homology of $\SL_2(F)$ are closely related.
In fact, there is a  canonical map $\alpha: H_3(\SL_2(F),\z)_\fff \arr K_3(F)^\ind$.
Suslin has raised the question that, is $\alpha$ an isomorphism \cite{sah1989}?
Recently Hutchinson and Tao
have shown that this map is surjective \cite{hutchinson-tao2009}. They also gave
some arguments about its injectivity.
In this article, we improve their arguments and show that $\alpha$ is bijective
if and only if the natural maps $H_3(\GL_2(F), \z)\arr H_3(\GL_3(F), \z)$ and
$H_3(\SL_2(F), \z)_\fff \arr H_3(\GL_2(F), \z)$ are injective.
\end{abstract}
\maketitle

\section*{Introduction}

For an infinite field $F$, Suslin has proved that the Hurewicz homomorphism
\[
h_3: K_3(F)=\pi_3(B\SL(F)^+) \larr H_3(B\SL(F)^+, \z)\simeq H_3(\SL(F), \z)
\]
is surjective  with 2-torsion kernel. In fact, he has shown that $h_3$
sits in the exact sequence
\[
K_2(F) \overset {l(-1)}{\larr} K_3(F) \larr H_3(\SL(F), \z) \larr 0,
\]
where the homomorphism ${l(-1)}: K_2(F) {\arr} K_3(F)$ coincides with multiplication
by $l(-1) \in K_1(\z)$ \cite[Lemma 5.2, Corollary 5.2]{suslin1991}. Let
\[
\alpha:
H_0(\fff,H_3(\SL_2(F),\z))\arr K_3(F)^\ind
\]
be the composition of the following sequence of homomorphisms
\begin{gather*}
\hspace{-1 cm}
H_0(\fff, H_3(\SL_2(F), \z)) \overset{\inc_\ast}{\larr} H_3(\SL(F), \z)
\overset{{\bar{h}}_3^{-1}}{\overset{\simeq}{\larr}} K_3(F)/l(-1)K_2(F) \\
\hspace{+6.5 cm}
\overset{p}{\larr} K_3(F)^\ind := K_3(F)/K_3^M(F),
\end{gather*}
where $\inc_\ast$ is induced by the inclusion $\inc: \SL_2(F) \arr \SL(F)$,
and $p$ is induced by the inclusion $l(-1)K_2(F) \se \im(K_3^M(F)\arr K_3(F))$.
For algebraically closed fields, it was known that $\alpha$ is an isomorphism
\cite{bloch2000}, \cite{sah1989}. Following this, Suslin raised the following
question:\\
\\
{\bf Question} (Suslin).
Is it true that $H_0(\fff, H_3(\SL_2(F), \z))$ coincides with $K_3(F)^\ind$?
(See \cite[Question 4.4]{sah1989}).
\\

In other words, is $\alpha$
bijective for an arbitrary infinite field $F$?
This question is true after killing 2-power torsion elements (i.e. after tensoring the
both side of this map with $\z[1/2]$) or when $\fff=\fff^2=\{a^2|a \in \fff\}$
\cite[Proposition 6.4]{mirzaii-2008}.

Recently Hutchinson and Tao have proved that $\alpha$
is surjective  \cite[Lemma 5.1]{hutchinson-tao2009}. At the end of their article,
they also stated that the injectivity of $\alpha$ is equivalent to the following three
conditions:\\
\par 1) $\im(H_3(\SL_2(F), \z) \arr H_3(\SL(F), \z)) \simeq K_3(F)^\ind$,\\
\par 2) $H_3(\GL_2(F), \z)\arr H_3(\GL_3(F), \z)$ is injective,\\
\par 3) $H_0(\fff, H_3(\SL_2(F), \z)) \arr H_3(\GL_2(F), \z)$ is injective.\\

Among these conditions, the first one seems very unnatural. It is also very
strong condition. In this article we show that this condition
can be dropped from this list. In fact we prove the following result.

\begin{athm}[{\bf Theorem.}]
Let $F$ be an infinite field. The following conditions are equivalent.\vspace{2mm}
\par {\rm (i)} The homomorphism $\alpha: H_0(\fff, H_3(\SL_2(F), \z)) \arr K_3(F)^\ind$
is bijective.\vspace{2mm}
\par {\rm (ii)} The natural homomorphisms $H_3(\GL_2(F), \z)\arr H_3(\GL_3(F), \z)$ and
$H_0(\fff, H_3(\SL_2(F), \z)) \arr H_3(\GL_2(F), \z)$ are injective.
\end{athm}

We should mention that in the mean time we also establish that the kernel of
the homomorphism
\[
H_3(\inc): H_3(\GL_2(F), \z)\arr H_3(\GL_3(F), \z)
\]
is equal to
\[
\im(H_3(\SL_2(F), \z) \arr H_3(\GL_2(F), \z)) \cap \fff \cup H_3(\GL_1(F), \z)),
\]
where the cup product is induced by the natural diagonal inclusion
$\inc: \fff \times \GL_1(F) \arr \GL_2(F)$. In another article
we have shown that the kernel of $H_3(\inc)$ is a 2-torsion  group \cite{mirzaii-2011}
(see Theorem \ref{kernel} below).

It seems that, for an arbitrary field, not much is known about the kernel of
\[
H_0(\fff, H_3(\SL_2(F), \z)) \arr H_3(\GL_2(F), \z),
\]
except that it is a 2-power torsion group
(see proof of Theorem 6.1 in \cite{mirzaii-2008}).

\subsection*{Notation}
In this article by $H_i(G)$ we mean  the  homology of group $G$
with integral coefficients, namely $H_i(G, \z)$.
By $\GL_n$ (resp. $\SL_n$) we mean the general (resp. special) linear
group $\GL_n(F)$ (resp. $\SL_n(F)$), where  $F$ is an infinite field.
If $A \arr A'$ is a homomorphism of abelian groups, by $A'/A$ we
mean $\coker(A \arr A')$ and we take other liberties of this kind.
Here by $\Sigma_n$ we mean the symmetric group of rank $n$.

\section{The group $H_1(\fff, H_2(\SL_2))$}

We start this section by looking at the corresponding Lyndon/Hochschild-Serre spectral
sequence of the commutative diagram of extensions
\[
\begin{CD}
1 \larr & \SL_2 & \larr & \GL_2 & \overset{\det}{\larr}  & \fff& \larr 1 \\
 &  @VVV  @VVV  @VVV & \\
1 \larr & \SL_3 & \larr & \GL_3 & \overset{\det}{\larr} & \fff& \larr 1.
\end{CD}
\]
So we get a morphism of spectral sequences
\[
\begin{CD}
\E_{p, q}^2=H_p(\fff, H_q(\SL_2)) &\Longrightarrow &H_{p+q}(\GL_2),\\
 @VVV  @VVV \\
E_{p, q}^2=H_p(\fff, H_q(\SL_3)) &\Longrightarrow &H_{p+q}(\GL_3).
\end{CD}
\]

By an easy analysis of this spectral sequence we obtain the
following commutative diagram with exact rows
\[
\begin{CD}
  & H_3(\SL_2)_\fff & \larr & H_3(\GL_2)/H_3(\GL_1) & \overset{\varphi}{\larr}
  & H_1(\fff, H_2(\SL_2)) \larr 0 \\
 &  @VVV  @VVV @VVV & \\
0 \larr & H_3(\SL_3)_\fff & \larr & H_3(\GL_3)/ H_3(\GL_1) & \overset{\psi}{\larr}
& H_1(\fff, H_2(\SL_3))  \larr 0.
\end{CD}
\]

The following theorem is due to Hutchinson and Tao
\cite[Theorem 3.2]{hutchinson-tao2009}, which is very fundamental in their
proof of the surjectivity of $\alpha$.

\begin{thm}\label{hutch-tao}
 The inclusion $\SL_2 \larr \SL_3$ induces a short exact sequence
\begin{gather*}\label{hutch-tao1}
0 \larr H_1(\fff, H_2(\SL_2))\larr H_1(\fff, H_2(\SL_3)) \larr k_3^M(F) \larr 0,
\end{gather*}
where $k_3^M(F):=K_3^M(F)/2$.
\end{thm}

Since the action of $\fff$ on $H_2(\SL_3)$
is trivial,
\[
H_1(\fff, H_2(\SL_3))\simeq \fff \otimes K_2^M(F).
\]
So we consider $H_1(\fff, H_2(\SL_2))$ as a subgroup of $\fff \otimes K_2^M(F)$.
It is easy to see that the map
\[
H_1(\fff, H_2(\SL_3)) \larr k_3^M(F)
\]
is induced by the natural product map $\fff \otimes K_2^M(F) \larr K_3^M(F)$.
Since  the $n$-th Milnor $K$-group, $K_n^M(F)$, is naturally isomorphic to the
$n$-th tensor of $\fff$ modulo the two families of relations
\[
a_1 \otimes \dots \otimes a_{n-1} \otimes (1-a_{n-1}), \ \ \ \ \ \ \
a_i \in  \fff, a_{n-1} \neq 1,
\]
\[
a_1 \otimes \cdots  a_i \otimes a_{i+1} \cdots \otimes a_n +
a_1 \otimes \cdots  a_{i+1} \otimes a_{i} \cdots \otimes a_n,
\ \  a_i \in  \fff,
\]
it easily follows that the kernel of the product map $\fff \otimes K_2^M(F)\larr K_3^M(F)$
is generated by elements $a \otimes \{b,c\} + b \otimes \{a, c\}$.
This proves the following lemma.

\begin{lem}\label{subgroup}
As a subgroup of
$H_1(\fff, H_2(\SL_3))=\fff \otimes K_2^M(F)$, the group $H_1(\fff, H_2(\SL_2))$
is generated by elements
$a \otimes \{b,c\} + b \otimes \{a, c\}$ and  $2d \otimes \{e,f\}$.
\end{lem}

To go further, we need to introduce some notations. Let $G$ be a group and set
\[
{\rm \bf{c}}({g}_1, {g}_2,\dots, {g}_n):=\sum_{\si \in S_n} {{\rm
sign}(\si)}[{g}_{\si(1)}| {g}_{\si(2)}|\dots|{g}_{\si(n)}] \in
H_n(G),
\]
where ${g}_i \in G$ pairwise commute and $S_n$ is the symmetric
group of degree $n$. Here we use the bar resolution of $G$
\cite[Chapter I, Section 5]{brown1994} to define the homology of $G$.

\begin{lem}
Let $G$ and $G'$ be two groups.
\par {\rm (i)} If $h_1\in G$ commutes with all the elements
$g_1, \dots, g_n \in G$, then
\[
{\rm \bf{c}}(g_1h_1, g_2,\dots, g_n)= {\rm \bf{c}}(g_1, g_2,\dots,
g_n)+{\rm \bf{c}}(h_1, g_2,\dots, g_n).
\]
\par {\rm (ii)}
For every $\sigma \in S_n$, ${\rm \bf{c}}(g_{\sigma(1)},\dots,
g_{\sigma(n)})={\rm sign(\sigma)} {\rm \bf{c}}(g_1,\dots,
g_n)$.
\par {\rm (iii)}
The cup product of ${\rm \bf{c}}(g_1,\dots, g_p)\in H_p(G)$
and ${\rm \bf{c}}(g_1',\dots, g_q') \in H_q(G')$ is ${\rm
\bf{c}}((g_1, 1), \dots, (g_p,1),(1,g_1'), \dots, (1,g_q')) \in
H_{p+q}(G \times G')$.
\end{lem}
\begin{proof}
The proof follows from direct computations, so we leave it to the
interested readers.
\end{proof}

\section{The kernel of $H_3(\GL_2)\larr H_3(\GL_3)$}

For simplicity, in the rest of this article, we use the following notation
\[
k_{a,b,c}:={\rm \bf{c}} (\diag(a,1), \diag(1,b),\diag(1,c)) \in H_3(\GL_2).
\]

The following theorem has been proved in \cite[Theorem 3.1]{mirzaii-2011}.

\begin{thm}\label{kernel}
The kernel of ${\inc_1}_\ast:H_3(\GL_2) \larr H_3(\GL_3)$
consists of elements of the form $\sum k_{a,b,c}+k_{b,a,c}$ such that
\[
\sum  a\otimes \{b, c\}+b\otimes \{a, c\}=0 \in \fff \otimes K_2^M(F).
\]
In particular $\ker({\inc_1}_\ast) \subseteq \fff \cup H_2(\GL_1) \subseteq H_3(\GL_2)$,
where the cup product is induced by the natural diagonal inclusion
$\fff \times \GL_1 \larr \GL_2$. Moreover $\ker({\inc_1}_\ast)$ is a $2$-torsion group.
\end{thm}

Let $\Psi$ and $\Phi$ be the following compositions,
\[
\fff \otimes K_2^M(F) \overset{\id_\fff \otimes \iota}{\larr}
\fff \otimes H_2(\GL_2) \overset{\cup}{\larr}
 H_3(\fff \times \GL_2) \overset{\inc_\ast}{\larr} H_3(\GL_3),
\]
\[
\fff \otimes K_2^M(F) \overset{\id_\fff \otimes \iota}{\larr}
\fff \otimes H_2(\GL_2) \overset{\cup}{\larr}
 H_3(\fff \times \GL_2) \overset{\beta_\ast}{\larr} H_3(\GL_2),
\]
respectively, where $\iota:K_2^M(F)\simeq H_2(\SL_2)_\fff \larr  H_2(\GL_2)$
is the natural inclusion given by the formula
$\{a,b\} \mt  {\rm \bf{c}} (\diag(a,1),\diag(b,b^{-1}))$ \cite[Proposition A.11]{elbaz1998}
and $\beta: \fff \times \GL_2 \larr \GL_2$  is given by $(a, A) \mt aA$.
It is easy to see that
\[
\Psi(a \otimes \{b,c\})={\rm \bf{c}} (\diag(a,1,1), \diag(1,b,1),\diag(1,c,c^{-1})),
\]
\[
\hspace{-1.1 cm}
\Phi(a \otimes \{b, c\})={\rm \bf{c}} (\diag(a,a), \diag(b,1),\diag(c,c^{-1})).
\]

\begin{lem}\label{theta}
Let $\Theta$ be the composition
\[
H_3(\GL_2)/H_3(\GL_1) \overset{\varphi}{\larr} H_1(\fff, H_2(\SL_2))
\hookrightarrow \fff\otimes K_2^M(F).
\]
Then\vspace{1.5mm}
\par {\rm (i)} $\Theta(k_{a,b,c}+k_{b,a,c})= a\otimes \{b,c\}+b \otimes \{a,c\}$,\vspace{1.5mm}
\par {\rm (ii)} $\Theta({\bf c}(\diag(a, a),\diag(b,1), \diag(c, c^{-1})))=
2a \otimes\{b,c\}$,\vspace{1.5mm}
\par {\rm (iii)} $\Theta(k_{c,a,b})= b\otimes \{a,c\}-a \otimes \{b,c\}$.
\end{lem}
\begin{proof}
(i) It is easy to see that the exact sequence
\[
0 \larr H_3(\SL_3)_\fff\larr H_3(\GL_3)/H_3(\GL_1) \overset{\psi}{\larr}
H_1(\fff, H_2(\SL_3)) \larr 0
\]
splits canonically by the map
\[
\fff \otimes K_2^M(F)=H_1(\fff, H_2(\SL_3)) \larr H_3(\GL_3)/H_3(\GL_1)
\]
defined by $\Psi$. Now consider the commutative diagram
\[
\begin{CD}
& H_3(\GL_2)/H_3(\GL_1) & \overset{\varphi}{\larr}
  & H_1(\fff, H_2(\SL_2))  \\
 &  @VVV  @VVV & \\
 & H_3(\GL_3)/ H_3(\GL_1) & \overset{\psi}{\larr}
& H_1(\fff, H_2(\SL_3)).
\end{CD}
\]
We have
\[
\begin{array}{rl}
\vspace{1.5mm}
{\inc_1}_\ast(k_{a,b,c}) \!\!\!\!
&={\bf c}(\diag(a,1,1),\diag(1,b,1), \diag(1,c, 1))\\
\vspace{1.5mm}
&={\bf c}(\diag(a,1,1),\diag(1,b,1), \diag(1,c, c^{-1}))\\
\vspace{1.5mm}
&+{\bf c}(\diag(a,1,1),\diag(1,b,1), \diag(1,1,c))\\
\vspace{1.5mm}
&={\bf c}(\diag(a,1,1),\diag(1,b,1), \diag(1,c, c^{-1}))\\
\vspace{1.5mm}
&-{\bf c}(\diag(b,1,1),\diag(1,a,1), \diag(1,1,c))\\
\vspace{1.5mm}
&={\bf c}(\diag(a,1,1),\diag(1,b,1), \diag(1,c, c^{-1}))\\
\vspace{1.5mm}
&-{\bf c}(\diag(b,1,1),\diag(1,a,1), \diag(1,c^{-1},c))\\
\vspace{1.5mm}
&-{\bf c}(\diag(b,1,1),\diag(1,a,1), \diag(1,c, 1))\\
&=\Psi(a \otimes \{b,c\}+b \otimes \{a,c\})-{\inc_1}_\ast(k_{b,a,c}).
\end{array}
\]
Hence ${\inc_1}_\ast(k_{a,b,c}+k_{b,a,c})=\Psi(a \otimes \{b,c\}+
b \otimes \{a,c\})$. Therefore
\[
\begin{array}{rl}
\vspace{1.5mm}
\Theta(k_{a,b,c}+k_{b,a,c})\!\!\!\! &= \psi\circ {\inc_1}_\ast(k_{a,b,c}+k_{b,a,c})\\
\vspace{1.5mm}
&=\psi\circ\Psi(a \otimes \{b,c\}+b \otimes \{a,c\})\\
&= a \otimes \{b,c\}+b \otimes \{a,c\}.
\end{array}
\]

(ii) Consider the composition
\[
\fff \otimes K_2^M(F) \overset{\Phi}{\larr} H_3(\GL_2)/H_3(\GL_1)
\overset{\Theta}{\larr} \fff \otimes K_2^M(F).
\]
The image of $\Phi(a\otimes \{b,c\})=
{\bf c}(\diag(a, a),\diag(b,1), \diag(c, c^{-1}))$
in the group $H_3(\GL_3)/H_3(\GL_1)=H_3(\SL_3)_\fff \oplus \fff \otimes K_2^M(F)$
is equal to
\begin{gather*}
\begin{array}{rl}
\vspace{1.5mm}
{\inc_1}_\ast\circ\Phi(a\otimes \{b,c\})\!\!\!\!&=
{\bf c}(\diag(a, a,1),\diag(b,1,1), \diag(c, c^{-1},1))\\
\vspace{1.5mm}
&={\bf c}(\diag(a, a, a^{-2}),\diag(b,1,1), \diag(c, c^{-1},1))\\
\vspace{1.5mm}
& \ +{\bf c}(\diag(1, 1,a^2),\diag(b,1,1), \diag(c, c^{-1},1))\\
\vspace{1.5mm}
&={\bf c}(\diag(a, 1, a^{-1}), \diag(b, 1, b^{-1}),\diag(c, c^{-1},1))\\
\vspace{1.5mm}
&\ +{\bf c}(\diag(a^2, 1, 1), \diag(1, b, 1),\diag(1, c, c^{-1})).
\end{array}
\end{gather*}
Therefore
$\Theta\circ \Phi(a\otimes \{b,c\})=\psi\circ{\inc_1}_\ast \circ\Phi(a\otimes \{b,c\})
=2 a\otimes \{b,c\}$.

(iii) First note that
\[
\begin{array}{rl}
\vspace{1.5mm}
\Phi(a \otimes \{b,c\}) \!\!\!\!
&={\bf c}(\diag(a,a),\diag(b,1), \diag(c, c^{-1}))\\
\vspace{1.5mm}
&={\bf c}(\diag(a,1),\diag(b,1), \diag(c,1))\\
&\ \ \ -k_{c,a,b}+ k_{a,b,c}+k_{b,a,c}.
\end{array}
\]
Therefore
\[
\begin{array}{rl}
\vspace{1.5mm}
\Theta(k_{c,a,b}) \!\!\!\! &=\Theta(k_{a,b,c}+k_{b,a,c})- \Theta(\Phi(a \otimes \{b,c\}))\\
&=b\otimes \{a,c\}-a \otimes \{b,c\}.
\end{array}
\]
\end{proof}

\begin{prp}\label{kernel2}
Let ${\inc_2}_\ast: H_3(\SL_2)_\fff \larr H_3(\GL_2)$ be induced by the natural
map $\inc_2:\SL_2 \larr \GL_2$. Then
\[
\im({\inc_2}_\ast)\cap \Big(\fff \cup H_2(\GL_1)\Big)=
\ker({\inc_1}_\ast).
\]
\end{prp}
\begin{proof}
By Theorem \ref{kernel}, the kernel of ${\inc_1}_\ast: H_3(\GL_2) \larr H_3(\GL_3)$
consists of elements of the form $\sum k_{a,b,c}+k_{b,a,c}$ such that
 \[
 \sum a\otimes \{b, c\} +b\otimes \{a, c\}=0 \in \fff \otimes K_2^M(F).
 \]
By Lemma \ref{theta}, we see that
\[
\varphi(\sum k_{a,b,c}+k_{b,a,c})= \sum a\otimes \{b, c\} +b\otimes \{a, c\}=0.
\]
Since the sequence
\[
H_3(\SL_2)_\fff \overset{{\inc_2}_\ast}{\larr}
H_3(\GL_2)/H_3(\GL_1) \larr H_1(\fff, H_2(\SL_2)) \larr 0,
\]
is exact, $\sum k_{a,b,c}+k_{b,a,c} \in \im({\inc_2}_\ast)$. Therefore
\[
\ker({\inc_1}_\ast) \se \im({\inc_2}_\ast)\cap \Big(\fff \cup H_2(\GL_1)\Big).
\]

Now let $x \in \im({\inc_2}_\ast)\cap \Big(\fff \cup H_2(\GL_1)\Big)$.
Then $x$ is of the following form
\[
x=\sum {\rm \bf{c}} (\diag(a_i,1), \diag(1,b_i),\diag(1,c_i)).
\]
Thus $\det_\ast(x)=\sum {\rm \bf{c}} (a_i, b_i,c_i)=0$, where
$\det_\ast: H_3(\GL_2) \larr H_3(\fff)$ is induced by the determinant.
By the inclusion $\bigwedge_\z^3 \fff \hookrightarrow H_3(\fff)$, we have
$a\wedge b \wedge c \mt {\rm \bf{c}}(a, b,c)$
(see for example \cite[Lemma 5.5]{suslin1991}). Thus
\[
\sum a_i \otimes b_i\otimes c_i=\sum a' \otimes a'\otimes b'+
\sum a'' \otimes b''\otimes a'' +\sum b''' \otimes a'''\otimes a'''.
\]
Under the composition
$\fff^{\otimes 3} \larr \fff \otimes H_2(\fff) \larr H_3(\GL_2)$
defined by
\[
 a \otimes b \otimes c \mt a \otimes {\rm \bf{c}} (b,c) \mt
 {\rm \bf{c}} (\diag(a,1), \diag(1,b),\diag(1,c))=k_{a,b,c},
\]
we see that $x$ has the following form
\[
x=\sum {\rm \bf{c}} (\diag(a,1), \diag(1,a),\diag(1,b)).
\]
For simplicity, we assume that
$x={\rm \bf{c}} (\diag(a,1), \diag(1,a),\diag(1,b))$. By Lemma \ref{theta},
$\Theta(x)= a\otimes \{a, b\} - b \otimes \{a, a\}=0$. Thus
\[
\begin{array}{c}
\vspace{1.5mm}
{\rm \bf{c}} (\diag(a,1,1), \diag(1,a,1),\diag(1,b, b^{-1}))\\
\vspace{1.5mm}
=\Psi(a\otimes \{a, b\})=\Psi(b\otimes \{a, a\})=\\
{\rm \bf{c}} (\diag(b,1,1), \diag(1,a,1),\diag(1,a,a^{-1})),
\end{array}
\]
and so
\[
\begin{array}{c}
\vspace{1.5mm}
+{\rm \bf{c}} (\diag(a,1,1), \diag(1,a,1),\diag(1,b, 1))\\
\vspace{1.5mm}
-{\rm \bf{c}} (\diag(a,1,1), \diag(1,a,1),\diag(1,1, b))\\
\vspace{1.5mm}
=\\
-{\rm \bf{c}} (\diag(b,1,1), \diag(1,a,1),\diag(1,1,a)).
\end{array}
\]
Hence in $H_3(\GL_3)$ we have
\[
\begin{array}{ll}
\vspace{1.5mm}
{\inc_1}_\ast(x)\!\!\!\!& ={\rm \bf{c}} (\diag(a,1,1), \diag(1,a,1),\diag(1,b,1))\\
\vspace{1.5mm}
&={\rm \bf{c}} (\diag(a,1,1), \diag(1,a,1),\diag(1,1,b))\\
\vspace{1.5mm}
&\ -{\rm \bf{c}} (\diag(b,1,1), \diag(1,a,1),\diag(1,1, a))\\
&=0
\end{array}
\]
Therefore $x \in \ker({\inc_1}_\ast)$ and this completes the proof of the proposition.
\end{proof}

\section{The indecomposable part of $K_3(F)$}

Define the {\it pre-Bloch group} $\ppp(F)$ of $F$ as the quotient of
the free abelian group $Q(F)$ generated by symbols $[a]$, $a \in \fff-\{1\}$,
by the subgroup generated by elements of the form
\[
[a] -[b]+\bigg[\frac{b}{a}\bigg]-\bigg[\frac{1- a^{-1}}{1- b^{-1}}\bigg]
+ \bigg[\frac{1-a}{1-b}\bigg],
\]
where $a, b \in \fff-\{1\}$, $a \neq b$. Define
\[
\lambda': Q(F) \larr \fff \otimes \fff, \ \ \ \ [a] \mapsto a \otimes (1-a).
\]
By a direct computation, we have
\[
\lambda'\Big(
[a] -[b]+\bigg[\frac{b}{a}\bigg]-\bigg[\frac{1- a^{-1}}{1- b^{-1}}\bigg]
+ \bigg[\frac{1-a}{1-b}\bigg] \Big)
=a \otimes \bigg( \frac{1-a}{1-b}\bigg)+\bigg(\frac{1-a}{1-b}\bigg)\otimes a.
\]
Let $(\fff \otimes \fff)_\sigma :=\fff \otimes \fff/
\lan a\otimes b + b\otimes a: a, b \in \fff \ran$.
We denote the elements of $\ppp(F)$ and $(\fff \otimes \fff)_\sigma$
represented by $[a]$ and $a\otimes b$ again by $[a]$ and $a\otimes b$,
respectively. Thus we have a well-defined map
\[
\lambda: \ppp(F) \larr (\fff \otimes \fff)_\sigma, \ \ \
[a] \mapsto a \otimes (1-a).
\]
The kernel of $\lambda$ is called the {\it Bloch group} of $F$ and is
denoted by $B(F)$. Therefore we obtain the exact sequence
\[
0 \larr B(F) \larr \ppp(F) \larr (\fff \otimes \fff)_\sigma \larr K_2^M(F) \larr 0.
\]

The following remarkable theorem is due to Suslin \cite[Theorem 5.2]{suslin1991}.

\begin{thm}\label{bloch-wigner-suslin}
Let $F$ be an infinite field. Then we have the exact sequence
\[
0 \larr \tors(\mu(F), \mu(F))^\sim \larr K_3(F)^\ind \larr B(F) \larr 0,
\]
where $\tors(\mu(F), \mu(F))^\sim $ is the unique nontrivial extension of
the group $\tors(\mu(F), \mu(F))$ by $\z/2$ if $\char(F)\neq 2$ and is equal to
$\tors(\mu(F), \mu(F))$ if $\char(F)= 2$.
\end{thm}

The following theorem has been proved in \cite[Theorem 4.4]{mirzaii-yeganeh2011}.

\begin{thm}\label{bloch-wigner-mirzaii-yeganeh}
Let $F$ be an infinite field. Then we have the exact sequence
\[
0 \larr \tors(\mu(F), \mu(F))^\sim \larr \H_3(\SL_2(F)) \larr B(F) \larr 0,
\]
where $\H_3(\SL_2(F)):=H_3(\GL_2)/(H_3(\GL_1) + \fff \cup H_2(\GL_1))$.
\end{thm}

These two theorems
suggest that $K_3(F)^\ind$ and $\H_3(\SL_2(F))$ should be isomorphism.
But there is no natural homomorphism from one of these groups to the other one!
But there is a natural maps from $H_3(\SL_2)_\fff$ to both of them.
Hutchinson and Tao have proved that $H_3(\SL_2)_\fff  \larr K_3(F)^\ind$
is surjective \cite[Lemma 5.1]{hutchinson-tao2009}. The next lemma claims that
this is also true for the other map.

\begin{lem}\label{sl2}
The map $\varsigma:H_3(\SL_2)_\fff \larr \H_3(\SL_2)$, induced by the natural map
$\SL_2 \larr \GL_2$, is surjective.
\end{lem}
\begin{proof}
Consider the exact sequence
\[
\begin{CD}
H_3(\SL_2)_\fff \larr H_3(\GL_2)/H_3(\GL_1) \overset{\varphi}{\larr}
H_1(\fff, H_2(\SL_2)) \larr 0.
\end{CD}
\]
By Lemma \ref{theta}, we have
\begin{gather*}
\Theta(k_{a,b,c}+k_{b,a,c}) =a\otimes \{b,c\} +b \otimes \{a,c\}, \\
\Theta({\bf c}(\diag(a, a),\diag(b,1), \diag(c, c^{-1})))=2a\otimes \{b,c\}.
\end{gather*}
Since $H_1(\fff, H_2(\SL_2))$ as a subgroup of
$H_1(\fff, H_2(\SL_3))=\fff \otimes K_2^M(F)$
is generated by elements $a\otimes \{b,c\} +b \otimes \{a,c\}$ and
$2d\otimes \{e,f\}$ and since $a\cup {\bf c}(b, c)=k_{a,b,c}$
vanish in $\H_3(\SL_2)$,  $H_3(\SL_2)_\fff \larr \H_3(\SL_2)$ must be surjective.
\end{proof}

Now we are ready to prove our main theorem.

\begin{thm}
Let $F$ be an infinite field. The following conditions are equivalent.
\par {\rm (i)} The homomorphism $\alpha: H_3(\SL_2)_\fff \larr K_3(F)^\ind$ is bijective.
\par {\rm (ii)} The natural homomorphisms ${\inc_1}_\ast: H_3(\GL_2) \larr H_3(\GL_3)$ and
${\inc_2}_\ast:H_3(\SL_2)_\fff \larr H_3(\GL_2)$ are injective.
\end{thm}
\begin{proof}
(ii) $\Rightarrow$ (i) Consider the surjective map
$\varsigma: H_3(\SL_2)_\fff \larr \H_3(\SL_2)$ from Lemma \ref{sl2}.
Let $\varsigma(x)=0$.
Then ${\inc_2}_\ast(x) \in \im({\inc_2}_\ast) \cap \fff \cup H_3(\GL_1)$.
But by Proposition \ref{kernel2} and the assumptions
\[
\im({\inc_2}_\ast) \cap \fff \cup H_3(\GL_1)=\ker({\inc_1}_\ast)=0.
\]
From this we have  ${\inc_2}_\ast(x)=0$ and hence $x=0$. Therefore $\varsigma$
is an isomorphism.
Now the claim follows by comparing the exact sequence of Theorem
\ref{bloch-wigner-mirzaii-yeganeh} and
Suslin's Bloch-Wigner exact sequence in Theorem \ref{bloch-wigner-suslin}.

(i) $\Rightarrow$ (ii) Let $\bar{F}$ be the algebraic closure of $F$. By a theorem
of Merkurjev and Suslin, $K_3(F)^\ind \larr K_3(\bar{F})^\ind$ is injective
\cite[Proposition 11.3]{merkurjev-suslin1990}. Thus from the commutative diagram
\[
\begin{CD}
H_3(\SL_2)_\fff\!\! @>>> \!\! H_3(\SL_2(\bar{F}))\\
   @VVV  @VV{\simeq}V  \\
K_3(F)^\ind \!\! @>>> \!\!K_3(\bar{F})^\ind,
\end{CD}
\]
and the injectivity of $\alpha$,
we deduce the injectivity of the map $H_3(\SL_2)_\fff \larr H_3(\SL_2(\bar{F}))$.
Now the injectivity of $H_3(\SL_2)_\fff \larr H_3(\GL_2)$ follows from
the injectivity of $H_3(\SL_2(\bar{F})) \larr H_3(\GL_2(\bar{F}))$
\cite[Theorem 6.1]{mirzaii-2008} and commutativity of the diagram
\[
\begin{CD}
H_3(\SL_2)_\fff\!\! @>>>\!\! H_3(\SL_2(\bar{F})) \\
   @VVV  @VVV  \\
H_3(\GL_2)\!\! @>>> \!\! H_3(\GL_2(\bar{F})).
\end{CD}
\]
On the other hand, by Proposition \ref{kernel2},
$\ker({\inc_1}_\ast) \se H_3(\SL_2)_\fff \se H_3(\GL_2)$.
Let ${\inc_1}_\ast(x)=0$.  It easily follows from the commutative diagram
\[
\begin{CD}
H_3(\SL_2)_\fff\!\! @>>> \!\!H_3(\GL_2) \!\!@>>> \!\!H_3(\GL_3)\\
   @VVV  @VVV @VVV  \\
H_3(\SL_2(\bar{F}))\!\! @>>>\!\! H_3(\GL_2(\bar{F})) \!\!@>>>\!\! H_3(\GL_3(\bar{F})).
\end{CD}
\]
that $x \in \ker\Big(H_3(\GL_2(\bar{F})) \larr H_3(\GL_3(\bar{F}))\Big)=0$.
Therefore $x=0$ \cite[Theorem 5.4(iii)]{mirzaii-2008}.
\end{proof}

\begin{rem}
(i) From the spectral sequence $\ee_{p,q}^2$, one get the exact sequence
\[
\vspace{3mm}
{\hspace{-2.5cm}
H_4(\GL_2)/H_4(\GL_1) \larr H_2(\fff, H_2(\SL_2)) \overset{\DD_{2,2}^2}{\larr} H_3(\SL_2)_\fff}
\]
\[
\hspace{3.2cm}
\larr  H_3(\GL_2)/ H_3(\GL_1) \larr H_1(\fff, H_2(\SL_2)) \larr 0.
\]
Thus the injectivity of ${\inc_2}_\ast$ is equivalent to
triviality of the differential $\DD_{2,2}^2$.
\par (ii)
Theorem \ref{kernel} gives a clear description of elements of the kernel of
${\inc_1}_\ast$. But there is no such information about the kernel of ${\inc_2}_\ast$.
It is easy to see that $s_{a,b,c}:={\rm \bf{c}} (\diag(a,a^{-1}), \diag(b,b^{-1}),\diag(c,c^{-1}))$
is in the kernel of ${\inc_2}_\ast$ and is 2-torsion:
\[
\begin{array}{rl}
\vspace{1.5mm}
s_{a,b,c}\!\!\!\!&={\rm \bf{c}} (\diag(a,a^{-1}), \diag(b,b^{-1}),\diag(c,c^{-1}))\\
\vspace{1.5mm}
&={\rm \bf{c}} (w.\diag(a,a^{-1}).w^{-1}, w.\diag(b,b^{-1}).w^{-1},w.\diag(c,c^{-1}).w^{-1})\\
\vspace{1.5mm}
& ={\rm \bf{c}} (\diag(a^{-1},a), \diag(b^{-1},b),\diag(c^{-1},c))\\
\vspace{1.5mm}
& =-s_{a,b,c},
\end{array}
\]
where $w:={\mtx 0 {-1} 1 0}$. But it is not clear to us why it should be zero.
It is not difficult to see that $\ker({\inc_2}_\ast)$ is a 2-power torsion group
(see for example the proof of Theorem 6.1 in \cite{mirzaii-2008}).
\end{rem}


\bigskip
\address{{\footnotesize

Department of Mathematics,

Institute for Advanced Studies in Basic Sciences,

P. O. Box. 45195-1159, Zanjan, Iran.

email:\ bmirzaii@iasbs.ac.ir

}}
\end{document}